\documentclass[twoside,leqno,11pt, A4]{amsart}
\usepackage{amsfonts}
\usepackage{amsmath}
\usepackage{amscd}
\usepackage{amssymb}
\usepackage{amsthm}
\usepackage{amsrefs}
\usepackage{latexsym}
\usepackage{mathrsfs}
\usepackage{bbm}
\usepackage{enumerate}
\usepackage{color}
\begin{document}

\newtheorem{theorem}[subsection]{Theorem}
\newtheorem{proposition}[subsection]{Proposition}
\newtheorem{lemma}[subsection]{Lemma}
\newtheorem{corollary}[subsection]{Corollary}
\newtheorem{conjecture}[subsection]{Conjecture}
\newtheorem{prop}[subsection]{Proposition}
\numberwithin{equation}{section}
\newcommand{\mr}{\ensuremath{\mathbb R}}
\newcommand{\mc}{\ensuremath{\mathbb C}}
\newcommand{\dif}{\mathrm{d}}
\newcommand{\intz}{\mathbb{Z}}
\newcommand{\ratq}{\mathbb{Q}}
\newcommand{\natn}{\mathbb{N}}
\newcommand{\comc}{\mathbb{C}}
\newcommand{\rear}{\mathbb{R}}
\newcommand{\prip}{\mathbb{P}}
\newcommand{\uph}{\mathbb{H}}
\newcommand{\fief}{\mathbb{F}}
\newcommand{\majorarc}{\mathfrak{M}}
\newcommand{\minorarc}{\mathfrak{m}}
\newcommand{\sings}{\mathfrak{S}}
\newcommand{\fA}{\ensuremath{\mathfrak A}}
\newcommand{\mn}{\ensuremath{\mathbb N}}
\newcommand{\mq}{\ensuremath{\mathbb Q}}
\newcommand{\half}{\tfrac{1}{2}}
\newcommand{\f}{f\times \chi}
\newcommand{\summ}{\mathop{{\sum}^{\star}}}
\newcommand{\chiq}{\chi \bmod q}
\newcommand{\chidb}{\chi \bmod db}
\newcommand{\chid}{\chi \bmod d}
\newcommand{\sym}{\text{sym}^2}
\newcommand{\hhalf}{\tfrac{1}{2}}
\newcommand{\sumstar}{\sideset{}{^*}\sum}
\newcommand{\sumprime}{\sideset{}{'}\sum}
\newcommand{\sumprimeprime}{\sideset{}{''}\sum}
\newcommand{\shortmod}{\ensuremath{\negthickspace \negthickspace \negthickspace \pmod}}
\newcommand{\V}{V\left(\frac{nm}{q^2}\right)}
\newcommand{\sumi}{\mathop{{\sum}^{\dagger}}}
\newcommand{\mz}{\ensuremath{\mathbb Z}}
\newcommand{\leg}[2]{\left(\frac{#1}{#2}\right)}
\newcommand{\muK}{\mu_{\omega}}
\newcommand{\seteq}{:=}
\newcommand{\odd}{\mathrm{\ primary}}
\newcommand{\res}{\mathrm{Res}}

\title[Large values of quadratic Hecke $L$-functions with prime-related moduli  of the Gaussian field]{Large values of quadratic Hecke $L$-functions with prime-related moduli  of the Gaussian field}

%%\date{\today}
\author[P. Gao]{Peng Gao}
\address{School of Mathematical Sciences, Beihang University, Beijing 100191, China}
\email{penggao@buaa.edu.cn}

\begin{abstract}
In this paper, we apply the 
resonance  method to exhibit large values at the central point of the family of quadratic Hecke $L$-functions with prime-related moduli of the Gaussian field under the generalized Riemann hypothesis. 
\end{abstract}

\maketitle



\section{Introduction}

The resonance method was induced by K. Soundararajan in \cite{Sound08} to study extreme values of $L$-functions. It is shown in \cite[Theorem 1, 2]{Sound08} that for 
sufficiently large $T$ and $X$, 
\begin{align*}
%%\label{zetalargevalue}
\max_{T\leq t\leq 2T}|\zeta( \frac 12+it)| & \geq \exp\left((1+o(1))\sqrt{\frac{\log T}{\log_2 T}}\right), \\
\label{Llargevalue}
\max_{\substack{d \text{ fundamental discriminant }\\X<|d|\leq 2X}}|L(\frac 12, \chi_d)| & \geq \exp\left(\left(\frac 1{\sqrt{5}} +o(1)\right)\sqrt{\frac{\log X }{\log_2 X}}\right),
\end{align*}
  where $\zeta(s)$ is the Riemann zeta function, and $L(s, \chi_d)$ is the $L$-function associated to the Kronecker symbol $\chi_d :=\leg {d}{\cdot}$. Here and in what follows, we denote $\log_j$ the $j$-fold iterated logarithm. 
  
   Using the long resonators introduced by C. Aistleitner in \cite{Aistleitner16} that relates the study of large values $\zeta(s)$ to the estimation of certain GCD sums,  A. Bondarenko and K. Seip \cite{BS17} proved that
\[
\max_{0\leq t\leq T}|\zeta(\frac 12+it)|\geq \exp\left((\frac 1{\sqrt{2}}+o(1))\sqrt{\frac{\log T \log_3 T}{\log_2 T}}\right). 
\] 
  The constant $1/\sqrt{2}$ was subsequently improved by the same authors to $1$ in \cite{BS18} and to $\sqrt{2}$ by R. de la Bret\`eche and G. Tenenbaum in \cite{BT19}.

 The long resonator technique of Bondarenko--Seip was further employed by P.  Darbar and G. Maiti in \cite{DM25} to study large values of $L(1/2, \chi_d)$ to show that under the generalized Riemann hypothesis (GRH), for sufficiently large $X$,  
	\[
	\max_{\substack{d \text{ fundamental discriminant }\\X<|d|\le 2X}}|L(1/2, \chi_{d})|\geq \exp\left(\left(\frac 1 2 +o(1) \right)\sqrt{\frac{\log X \log_3 X}{\log_2 X}}\right).
	\]
  The constant $1/2$ was improved to $1$ by Z. Dong, W. Wang, H. Zhang and S. Zhao in \cite{DWZ26}. 

  On the other hand,  M. Fan, S. Hua and S. Xie \cite{FHX26} studied large values for the family of quadratic Dirichlet $L$-
functions with prime moduli to show that for sufficiently large $X$, 
\begin{align*}
%%\label{p}
\max_{\substack{X< p \leq 2X \\ p \text{ prime}, \ p \equiv 1 \shortmod 8}}|L(\frac 12, \chi_p)|\geq \exp\left(\left(\sqrt{\frac 8{45}}+o(1)\right)\sqrt{\frac{\log X }{\log_2 X}}\right).
\end{align*}
  
  In \cite{Gao2026-1}, the author considered the family of quadratic Dirichlet $L$-functions with prime-related moduli to show that under GRH, for sufficiently large $X$, 
\begin{align}
\label{Llowerboundprime}
\begin{split}
  \max_{\substack{X< p \le 2X \\ p \text{ \rm prime }}}\Big|L(\frac 12, \chi_{8p})\Big |\geq \exp\left(\left(\frac 1 2 +o(1) \right)\sqrt{\frac{\log X \log_3 X}{\log_2 X}}\right).
\end{split}
\end{align}	
   The constant $1/2$ was improved to $1$ by Z. Dong, W. Wang, H. Zhang and S. Zhao in \cite{DWZ26-1} under GRH. 

 It is the aim of this paper to employ the ideas in \cites{BS18, Gao2026-1, DWZ26-1} to study large values at the central point of the family of quadratic Hecke $L$-functions of the Gaussian field under GRH. For this, we denote $K=\mq(i)$ the Gaussian field and $\mathcal O_K=\mz[i]$ its ring of integers. We say an element $q \in \mathcal O_K$ is a prime if the ideal generated by $q$ is a prime ideal. From now on, we denote $\chi_c$ for the quadratic symbol $\leg {c}{\cdot}$ defined in Section \ref{sect: Kronecker} for any $c \in \mathcal O_K$. Note that it is also shown there that $\chi_{(1+i)^5c}$ is a Hecke character for any $c \in \mathcal O_K$, and we denote $L(s, \chi_{(1+i)^5d})$ the associated Hecke $L$-function.
 
    Our main result investigates the maximum sizes of $L(s, \chi_{(1+i)^5q})$ for $q$ runs over prime elements in $\mathcal O_K$.  
\begin{theorem}
\label{main theorem 1}
 With the notation as above and assume the truth of GRH. We have for sufficiently large $X$, 
\begin{align}
\label{Llowerbound}
\begin{split}
  \max_{\substack{q \text{ \rm prime }\\X< N(q) \le 2X}}\Big|L(\frac 12, \chi_{(1+i)^5q})\Big |\geq \exp\left(\left(\sqrt{2} +o(1) \right)\sqrt{\frac{\log X \log_3 X}{\log_2 X}}\right).
\end{split}
\end{align}	
\end{theorem}

  We note that in \cite{DL26}, Z. Dong and L. Liu studied large values of quadratic Hecke $L$-functions over a general number field. Our result differs from theirs by restricting our family of Hecke $L$-functions to those with prime-related moduli. Moreover, our choice of the resonator is supported on elements in $\mathcal O_K$ that are not necessarily square-free, which allows us to achieve the constant $\sqrt{2}$ in \eqref{Llowerbound}.    We note that the approach in our proof may be carried out to show that the estimation given in \eqref{Llowerboundprime} holds with the constant $1/2$ being replaced by $\sqrt{2}$.

\section{Preliminaries}
\label{sec 2}

\subsection{Quadratic Hecke $L$-functions}
\label{sect: Kronecker}
  Recall that we denote $K=\mq(i)$ the Gaussian field and $\mathcal O_K=\mz[i]$ its ring of integers. 
  For any $c \in K$, we write $N(c)$ for its the norm. 
  We reserve the letter $\varpi$ for a prime in $\mathcal O_K$ throughout the paper.
  
    We say an element $n \in \mathcal O_K$ is odd if the ideal $(n,2)$ generated by $n$ and $2$ equals $\mathcal O_K$. For an odd prime $\varpi \in \mathcal{O}_K$, the quadratic symbol $\leg {\cdot}{\varpi}$ is defined for $a \in
\mathcal{O}_K$, $(a, \varpi)=1$ by $\leg{a}{\varpi} \equiv
a^{(N(\varpi)-1)/2} \pmod{\varpi}$, with $\leg{a}{\varpi} \in \{ \pm 1\}$. When $\varpi | a$, we define
$\leg{a}{\varpi} =0$.  We further define $\leg {\cdot}{c}=1$ for $c \in U_K$, where $U_K=\{ \pm 1, \pm i \}$ denotes the group of units of $\mathcal{O}_K$.  The quadratic character is then extended
to any composite odd $n$ multiplicatively.   

   For any $0 \neq c \in \mathcal O_K$, we follow \cite[Section 3.8]{iwakow} to define a Dirichlet character $\chi \pmod {c}$ to be a homomorphism
\begin{align*}
%%\label{chi}
  \chi: \left (\mathcal{O}_K / (c) \right )^{\times}  \rightarrow S^1 :=\{ z \in \mc :  |z|=1 \}.
\end{align*}
  We say that $\chi$ is primitive $\pmod {c}$ if it does not factor through $\left (\mathcal{O}_K / (c') \right )^{\times}$ for any divisor $c'$ of $c$ with $N(c')<N(c)$.

   If  $\chi(u)=1$ for any $u \in U_K$, then $\chi$ may be regarded as defined on ideals of $\mathcal O_K$ since every ideal is principal. In this case, we say that this $\chi$ is a Hecke character modulo $q$ of trivial infinite type. We also say a Hecke character is primitive if it is primitive as a Dirichlet character. The quadratic symbol $\leg {\cdot}{n}$ defined above is a Dirichlet character modulo $n$. 
As mentioned in the Introduction, it is shown in \cite[Section 2.1]{G&Zhao4} that for any $c \in \mathcal O_K$, the symbol $\chi_{(1+i)^5c}$ is a Hecke character $\pmod {(1+i)^5c}$ of trivial infinite type. Furthermore, when $c$ is square-free, $\chi_{(1+i)^5c}$ is non-principal and primitive, where we say an element $c \in \mathcal O_K$ is square-free if the ideal $(c)$ generated by it is not divisible by the square of any prime ideal. 

  We note the following supplementary law for primary $n=a+bi$ with $a, b \in \mz$ from \cite[Lemma 8.2.1]{BEW}.
\begin{align}
\label{supprule}
  \leg {i}{n}=(-1)^{(1-a)/2}.
\end{align}

  We observe that there exists a primitive quadratic Dirichlet character $\psi_2$ modulo $2$ as $\left (\mathcal{O}_K / (2) \right )^*$ is isomorphic to the cyclic group of order two generated by $i$.  Therefore we have $\psi_2(n)=-1$ for $n \equiv i \pmod 2$. 
  
We say an element $c \in \mathcal O_K$ is primary if $c \equiv 1 \pmod {(1+i)^3}$ and we recall from Lemma 7 of Section 7 in \cite[Chap. 9]{I&R} that that every ideal in $\mathcal{O}_K$ co-prime to $2$ has a unique primary generator congruent to $1$ modulo $(1+i)^3$.  It follows from Lemma
6 on \cite[p.121]{I&R} that an element $n=a+bi \in \mathcal{O}_K$ with $a, b \in \mz$ is primary if and only if $a \equiv 1 \pmod{4}, b \equiv
0 \pmod{4}$ or $a \equiv 3 \pmod{4}, b \equiv 2 \pmod{4}$.  We shall refer to these two cases above as $n$ is of type $1$ and type $2$, respectively. 
We now define for any primary $n \in \mathcal O_K$, 
\begin{align}
\label{psidef}
  \psi_n =
\begin{cases}
  \leg {\cdot}{n}, \quad  \text{if $n$ is of type $1$}, \\
  \psi_2 \cdot \leg {\cdot}{n}, \quad  \text{if $n$ is of type $2$}. 
\end{cases}
\end{align}
It follows from \eqref{supprule} that $\psi_n$ is also a Hecke character of trivial infinite type with modulus $n$ (resp. $2n$) when $n$ is of type $1$ (resp. of type $2$).

\subsection{Approximate functional equations}
\label{sect: FE}
  We note the following approximate functional equation for quadratic Hecke $L$-functions, taken from \cite[Lemma 2.7]{Gao24}. 
\begin{lemma}
\label{lem:AFE}
With the notation as above. We have for any odd, square-free $d \in \mathcal O_K$,  
\begin{align*}
%%\label{fcneqnL}
\begin{split}
 L(\frac 12 , \chi_{(1+i)^5d}) = & 2\sum_{\substack{n \odd }} \frac{\chi_{(1+i)^5d}(n)}{N(n)^{\frac{1}{2}}} V
\left(\frac{ N(n)}{N(d)^{1/2}} \right),
\end{split}
\end{align*}
where 
\begin{align*}
%%\label{eq:Vdef}
 V(x) = \frac{1}{2 \pi i} \int\limits\limits_{(2)}  \left(\frac{2^{5/2}}{\pi}\right)^{s}
\frac {\Gamma(\frac{1}{2}+s)}{\Gamma(\frac{1}{2})} x^{-s} \frac{ds}{s}.
\end{align*}
\end{lemma}

  Note that similar to the proof of \cite[Lemma 2.1]{sound1}, one shows that the function $V(x)$ is real-valued, smooth on $(0, +\infty)$, bounded as $x$ approaches $0$ and decays exponentially as $x\to +\infty$. More precisely,  we have for any $\varepsilon>0$, 
\begin{equation} 
\label{2.07}
      V\left (x \right) = 1+O(x^{1/2-\varepsilon}) \; \mbox{for} \; 0<x <1   \quad \mbox{and} \quad V^{(j)}\left (x \right) =O(e^{-x}) \; \mbox{for}
      \; x >0, \; j \geq 0.
\end{equation}

   We also note that
\begin{align}
\label{eq:Vder}
 V'(x) = -\frac{1}{2 \pi i} \int\limits\limits_{(2)}  \left(\frac{2^{5/2}}{\pi}\right)^{s}
\frac {\Gamma(\frac{1}{2}+s)}{\Gamma(\frac{1}{2})} x^{-s-1} ds.
\end{align}   
   Recall that for $\Re(s)>0$,  
\begin{align*}
%%\label{Gamma}
  \Gamma(s)=\int^{\infty}_0e^{-x}x^{s-1}dx.
\end{align*}      
  It follows from the inverse Mellin transformation that for $x>0$, 
\begin{align*}
%%\label{Gammainv}
  e^{-x}=\frac{1}{2 \pi i} \int\limits\limits_{(2)}\Gamma(s)x^{-s}ds. 
\end{align*}
  We deduce from the above that
\begin{align*}
%%\label{Gammainv1}
  x^{1/2}e^{-x}=\frac{1}{2 \pi i} \int\limits\limits_{(2)}\Gamma(s+\frac 12)x^{-s}ds. 
\end{align*} 
  It follows from  \eqref{eq:Vder}, the observation that $\Gamma(1/2)=\sqrt{\pi}$ (see \cite[(C.7)]{MVa1}) and the above that we have $V'(x)<0$ for $x>0$. Note moreover that by \eqref{2.07} we have $\lim_{x \rightarrow \infty}V(x)=0$. Thus we conclude that 
\begin{align}
\label{Vsign}
  V(x) \geq 0, \quad x \geq 0.
\end{align} 

\subsection{Smoothed character sums}
\label{smoothsum}

   We denote $\square$ for a perfect square in $\mathcal O_K$. We define $\delta_{c=\square}=1$ if $c$ is a perfect square and $\delta_{c=\square}=0$ otherwise.  In the remainder of the paper, let $\Phi$ denote a smooth, non-negative function compactly supported on $[1,2]$ satisfying $\Phi(x) =1$ for $x\in [5/4,7/4]$. The Mellin transform of $\Phi(x)$ is wirtten as ${\widehat \Phi}(s)$ so that for any complex number $s$,
\begin{equation*}
%%\label{Phicheck}
{\widehat \Phi}(s) = \int\limits_{0}^{\infty} \Phi(x)x^{s}\frac {\dif x}{x}.
\end{equation*}
   We have the following result on the smoothed quadratic Hecke character sums.
\begin{lemma}
\label{lemma logd}
With the notation as above and assuming the truth of GRH. Let $c \in \mathcal O_K$ be primary and $\Phi(X)$ be a smooth function fitting the above descriptions. Then for any $\varepsilon>0$,
\begin{equation} 
\label{wsum}
 \sum_{q \odd} (\log N(q)) \chi_{(1+i)^5q}(c) \Phi \left( \frac {N(q)}X \right) =  \delta_{c=\square}\widehat{\Phi}(1)X+O \left( X^{1/2+\varepsilon}\log
  (N(c)+2) \right).
 \end{equation}
\end{lemma}
\begin{proof}
  Our proof is similar to that for the quadratic Dirichlet character case given in \cite[Lemma 2.4]{G&Zhao23-01}. As $\Phi$ is compactly supported, we obtain that
\begin{align}
\label{sumlambda}
\begin{split}
 & \sum_{q \odd} (\log N(q)) \chi_{(1+i)^5q}(c) \Phi \left( \frac {N(q)}X \right) \\
 =& \sum_{n \odd}  \chi_{(1+i)^5n}(c) \Lambda_{K}(n) \Phi \left( \frac {N(n)}X \right)
 +O \left(\sum_{\substack{N(q)^j \leq X^{1+\varepsilon}, \ j \geq 2 }} (\log N(q))\Phi \left( \frac {N(q)^j}X \right)   \right ).
\end{split}
\end{align}

Now by \cite[(5.49)]{iwakow} and \cite[Theorem 5.15]{iwakow}, we have under GRH, 
\begin{align}
\label{ppowerest}
 \sum_{\substack{p^j \leq X^{1+\varepsilon}, \ j \geq 2}}(\log N(q))\Phi \left( \frac {N(q)^j}X \right) \ll
 X^{\varepsilon}\sum_{\substack{N(q) \leq X^{1/2+\varepsilon}}} \log N(q) \ll X^{1/2+\varepsilon}.
\end{align}

   Next, we apply Mellin inversion to see that
\begin{align}
\label{int}
\sum_{n \odd}  \chi_{(1+i)^5n}(c) \Lambda_{K}(n) \Phi \left( \frac {N(n)}X \right)
=& -\frac {\chi_{(1+i)^5}(c)}{2\pi i}\int\limits_{(2)} \frac {L'(s, \psi_c)}{L(s, \psi_c)} \widehat{\Phi}(s)X^s \dif s,
\end{align}
 where we recall here that $\psi_c$ is defined in \eqref{psidef}. 

  Suppose that $\psi_c$ is induced by a primitive Hecke character $\widetilde \psi_c$, then we have
\begin{align*}
 L(s, \psi_c)=L(s, \widetilde \psi_c)\prod_{\varpi | c}(1-\widetilde \psi_c(\varpi)N(\varpi)^{-s}). 
\end{align*}
   It follows that
\begin{align*}
\frac {L'(s, \psi_c)}{L(s, \psi_c)}=\frac {L'(s, \widetilde \psi_c)}{L(s, \widetilde \psi_c)}+\sum_{\substack{ \varpi \odd \\ \varpi | c}}\frac {(\log N(\varpi))\psi_c(\varpi)N(\varpi)^{-s}}{(1-\widetilde \psi_c(\varpi)N(\varpi)^{-s})}. 
\end{align*}  
   
  We substitute the above into \eqref{int} and then evaluate the resulting integral by shifting the line of integration to $\Re(s)=1/2+\varepsilon$ for any $\varepsilon>0$.  There is a pole at $s=1$ with residue
  $-\widehat{\Phi}(1)X$ only if $c$ is a perfect square.  Note from our discussions from Section \ref{sect: Kronecker} that the modulus of $\widetilde \psi_c$ is a divisor of $4c$. Thus by 
  \cite[Theorem 5.17]{iwakow}, we have under GRH for $\Re(s) \geq 1/2+\varepsilon$,
\begin{align}
\label{Lderbound}
  \frac {L'(s, \widetilde \psi_c)}{L(s, \widetilde \psi_c)}  \ll \log \big ((N(c)+2)(1+|s|)\big).
\end{align}
   Moreover, we have for $\Re(s) \geq 1/2+\varepsilon$,
\begin{align}
\label{nonprimsum}
 \sum_{\substack{ \varpi \odd \\ \varpi | c}}\frac {(\log N(\varpi))\psi_c(\varpi)N(\varpi)^{-s}}{(1-\widetilde \psi_c(\varpi)N(\varpi)^{-s})} \ll \sum_{\substack{ \varpi \odd \\ \varpi | c}}\log N(\varpi) \ll \log \big (N(c)+2\big).
\end{align}     
   
  It follows from \eqref{Lderbound} and \eqref{nonprimsum} that the new integration on the line $\Re(s)=1/2+\varepsilon$ can be estimated as $O(X^{1/2+\varepsilon}\log  (N(c)+2))$ by using the rapid decay of $\widehat{\Phi}$ on the vertical line. This implies that
\begin{align*}
%%\label{int}
\sum_{n \odd}  \chi_{(1+i)^5n}(c) \Lambda_{K}(n) \Phi \left( \frac {N(n)}X \right)
=&  \delta_{c=\square}\widehat{\Phi}(1)X+O \left( X^{1/2+\varepsilon}\log
  (N(c)+2) \right).
\end{align*}  
  
  The expression given in \eqref{wsum} now follows from \eqref{sumlambda}, \eqref{ppowerest} and the above. This completes the proof of the lemma. 
\end{proof}

\subsection{Large GCD sums}
\label{gcdsum}
 
   Denote $\mathcal P$ the set of primary primes in $\mathcal O_K$ and let $M$ be a large integer.  Denote \((m,n)_K\) and \([m,n]_K\) the greatest common divisor and the least common multiple of two primary elements \(m, n \in \mathcal O_K\), respectively. Thus, if we write $m=\prod_{\varpi_i \odd}\varpi^{\alpha_i}_i, n=\prod_{\varpi_i \odd}\varpi^{\beta_i}_i$ with $\alpha_i, \beta_i$ being non-negative integers, then we have
\begin{align*}
%%\label{gcd}
  (m,n)_K=\prod_{\varpi_i \odd}\varpi^{\min(\alpha_i,\beta_i)}_i, \quad [m,n]_K=\prod_{\varpi_i \odd}\varpi^{\max(\alpha_i,\beta_i)}_i. 
\end{align*}   
    The above then implies that
\begin{align}
\label{gcdprod}
  (m,n)_K[m,n]_K=mn. 
\end{align}     
   
    In this section, we follow the treatments given in \cite[Section 2.2]{BT19} to construct certain set $\mathcal M$ such that the quantity
\begin{align*}
%%\label{defSa}
 S( \mathcal M):=\sum_{m,n\in\mathcal M} \frac{N((m,n)_K^2)}{\sqrt{N(m)N(n)}}
\end{align*}         
     is large.
     
   For this, we let $u\in (1,e]$, $a\in (1,+\infty)$ and $\gamma\in (0,1)$ be three parameters to be fixed later such that $a\gamma<1/\log u$.  Denote
 \begin{align}
\label{Ikdef}
  I_k:=\Big (u^k\log M\log_2M, u^{k+1}\log M\log_2M\Big]\qquad 1\leqslant k\leqslant (\log_2M)^\gamma. 
\end{align}
  
   For any real $x>0$, denote $\pi_K(x)$ to be the number of prime ideals in $\mathcal O_K$ whose norms do not exceed $x$.  It follows from the prime ideal theorem (see \cite[Theorem 8.9]{MVa1}) that 
\begin{align*}
%%\label{estPk}
 P_k :=|I_k\cap\mathcal P|=\pi_{K}(u^{k+1}\log M\log_2M)-\pi_{K}(u^{k}\log M\log_2M)
\leqslant  u^{k +1} \log M . 
\end{align*}
  Also, we have
$$ P_k = u^{k }(u-1)\log M \Big(1+O\Big(\frac {k+\log_3M}{\log_2M}\Big)\Big).$$
  Let $\lfloor x\rfloor$ denote the largest integer not exceeding $x$ for any real $x$ and let $\omega_{K}(n)$ denote the number of distinct prime factors of $n$ for any $n \in \mathcal O_K$. For $1\leqslant k\leqslant (\log_2M)^\gamma$,  we define 
\begin{align*}
%%\label{defJk}
 J_k:=2\lfloor \frac {a\log M}{2k^2\log_3M} \rfloor, \quad M_k:=\prod_{\substack{\varpi \in I_k \\ \varpi \odd}}\varpi. 
\end{align*} 
 
  We further define the sets 
\begin{align*}
%%\label{defJk}
\mathcal M_k  &:=\Big\{ m  : m=\frac \ell{v}M_k,\,\omega_{K}(\ell)\leqslant \frac {J_k}2,\,\omega_{K}(v)\leqslant \frac {J_k}{2}, \,\ell v\mid M_k\Big\}, \qquad 1\leqslant k\leqslant (\log_2M)^\gamma. 
\end{align*}   
  The set $\mathcal M$ we aim to construct is then given by
\begin{align}
\label{defM}
\mathcal M :=\Big\{ m=\prod_{1\leqslant  k\leqslant (\log_2M)^\gamma} m_k\,:\quad m_k\in \mathcal M_k \quad  1\leqslant k\leqslant (\log_2M)^\gamma\Big\}.
\end{align}   

   The following result estimates the size of $\mathcal M$.  
\begin{lemma}
\label{lemma Msize}
With the notation as above. For  $a<1/(\gamma \log u)$, we have $ |\mathcal M |\leqslant M$. Moreover, if $P_+(n)$ denotes the largest norm of the prime divisors of $n$, we have
$$y_{\mathcal M} :=\max_{m\in\mathcal M}P_+(m)\le (\log M)^{1+o(1)}.$$ 
\end{lemma}
\begin{proof}
  The proof of the first assertion of the lemma is similar to that of \cite[Lemme 2.2]{BT19}, so we omit it here. The second assertion of the lemma follows straightforwardly from the construction of $\mathcal M$. 
\end{proof} 

  Our next result estimates the number of perfect squares whose divisors are in $\mathcal M$.
\begin{lemma}
\label{lemma Msquaresize}
 With the notation as above. We have
\begin{align*}
%%\label{Msquaresize}
  \sum_{\substack{m, n \in \mathcal{M} \\ m n =\square}}1 \ll M. 
\end{align*} 
\end{lemma}
\begin{proof}
  We write $m=\prod_{1\leqslant  k\leqslant (\log_2M)^\gamma} m_k, n=\prod_{1\leqslant  k\leqslant (\log_2M)^\gamma} n_k$ with $m_k, n_k \in \mathcal M_k$ for $1\leqslant  k\leqslant (\log_2M)^\gamma$. Then we have $mn=\square$ if and only if $m_kn_k=\square$ for all $k$. We write $m_k=\frac {\ell_1}{v_1}M_k, n_k=\frac {\ell_2}{v_2}M_k$ and we note that we have $(\ell_i, v_i)=1, i=1,2$. It follows that, for fixed $l_1, v_1$, we have  $\ell_2=(\ell_1, \ell_2)_K \cdot v_1/(v_1, v_2)_K, v_2=(v_1, v_2)_K \cdot \ell_1/(\ell_1, \ell_2)_K$. Thus, $l_2, v_2$ are determined completely by $(\ell_1, \ell_2)$ and $(v_1, v_2)$, which are divisors of $\ell_1$ and $v_1$, respectively. As $\ell_1, v_1$ are square-free, it follows that for any fixed $m_k$, there are at most $2^{\omega_K(\ell_1)+\omega_K(v_1)}$ different choices of $n_k$. As $\omega_{K}(\ell_1),\omega_{K}(v_1)\leqslant \frac {J_k}{2}$, we see that the number of different $n_k$ is $\leq 2^{J_k}$ for any fixed $m_k$. We thus conclude that
\begin{align*}
%%\label{Msquaresize1}
  \sum_{\substack{m, n \in \mathcal{M} \\ m n =\square}}1 \ll \prod_{1\leqslant k\leqslant (\log_2M)^\gamma} 2^{J_k}|M_k|. 
\end{align*}   
   We then argue as in the proof of \cite[Lemme 2.2]{BT19} to see that the right-hand side expression above is $\ll M$. This completes the proof of the lemma. 
\end{proof}   

   Our last result in this section is an analogue to the lower bound given in \cite[Section 2.3]{BT19} for the quantity constructed in \cite[(2.6)]{BT19}, so we omit its proof here. 
\begin{lemma}
\label{GCD}
    With the notation as above. By choosing $u$ and $a\gamma\log u$ close to $1$, we have as  $M\to\infty$, 
   $$S( \mathcal M) \geq M\exp\bigg((2+o(1))\sqrt{\frac{\log M\log_3M}{\log_2M}}\bigg).$$
\end{lemma}

\section{Proof of Theorem \ref{main theorem 1}}

  It suffices to establish Theorem \ref{main theorem 1} by restricting $q$ on primary primes. Let \(\mathcal M\) be the set given in \eqref{defM}. We define the resonator for $L(\frac 1 2, \chi_{(1+i)^5q})$ for any primary prime $q \in \mathcal O_K$ to be the Dirichlet polynomial 
\[
R_q:= \sum_{m\in \mathcal{M}}\chi_{(1+i)^5q}(m).
\] 
 
   Let $\Phi$ be the function  described in Section \ref{smoothsum}. We set 
\begin{align*} 
 \mathcal{S}_{1}:=\sum_{q \odd} (\log N(q))L(\frac{1}{2},\chi_{(1+i)^5q}) R_q^2 \Phi (\frac {N(q)}{X}), \quad \mathcal{S}_{2}:=\sum_{q \odd }(\log N(q))R_q^2\Phi (\frac {N(q)}{X}).
\end{align*}
  Note that by Lemma \ref{lem:AFE} and \eqref{Vsign} that we have $L(\frac{1}{2},\chi_{(1+i)^5q}) \in \mr$.  As we also have $R^2_q \geq 0$, it follows that 
\begin{align}  
\label{maxlower}
 \max_{\substack{X< N(q) \le 2X}}\Big|L(\frac 12, \chi_{(1+i)^5q})\Big |\ge \frac{\mathcal{S}_1}{\mathcal{S}_2}. 
\end{align}

It then suffices to establish a lower bound for $\mathcal{S}_{1}$ and an upper bound for $\mathcal{S}_{2}$. We first note from Lemma \ref{lem:AFE} that
\begin{align}
\label{S1eval}
\begin{split}
	\mathcal{S}_{1}=&\sum_{q \odd}(\log N(q))L(\frac{1}{2},\chi_{(1+i)^5q}) R_q^2 \Phi (\frac {N(q)}{X}) \\
	=&2\sum_{m, n\in \mathcal{M}}\sum_{l\odd } \frac{1}{\sqrt{N(l)}}\sum_{q \odd} (\log N(q))\chi_{(1+i)^5q}(lmn) V\left(\frac{N(l)}{\sqrt{N(q)}} \right)\Phi (\frac {N(q)}{X}).   
\end{split}  
\end{align}

  We now apply Lemma \ref{lemma logd} and partial summation to see that
\begin{align*}
\begin{split}
&	\sum_{q \odd}(\log N(q))\chi_{(1+i)^5q}(lmn) V\left(\frac{N(l)}{\sqrt{N(q)}} \right) \Phi (\frac {N(q)}{X}) \\
=& \int^{2X}_X V\left(\frac{N(l)}{\sqrt{t}} \right) d \left( t\delta_{lmn=\square}{\widehat \Phi}(1)+ O\left(t^{1/2+\varepsilon}\log
  (N(lmn)+2) \right)  \right) \\
=& {\widehat \Phi}(1)X \delta_{lmn=\square}\int^{2}_1 V\left(\frac{N(l)}{\sqrt{Xt}} \right) dt +V\left(\frac{N(l)}{\sqrt{t}} \right)O\left(t^{1/2+\varepsilon}\log
  (N(lmn)+2) \right)\Big |^{2X}_X\\
& -\int^{2X}_XO\left(t^{1/2+\varepsilon}\log
 (N(lmn)+2) \right)V'\left(\frac{N(l)}{\sqrt{t}}\right ) \frac {N(l)}{2t^{3/2}}dt \\
=: & {\widehat \Phi}(1)X  \delta_{lmn=\square}\int^{2}_1 V\left(\frac{N(l)}{\sqrt{Xt}} \right) dt+R. 
\end{split}
\end{align*}

   It follows from the above and \eqref{S1eval} that 
\begin{align*}	
    \mathcal{S}_{1}={\widehat \Phi}(1)X
     \sum_{m, n \in \mathcal{M}}& \sum_{\substack{l \odd \\lmn =\square}} \frac{1}{\sqrt{N(l)}}\int_{1}^{2}V\left(\frac{N(l)}{\sqrt{Xt}} \right) dt+O\Bigg(\sum_{m, n\in \mathcal{M}} \sum_{l\odd } \frac{1}{\sqrt{N(l)}}R\Bigg).
\end{align*}
    In view of the rapid decay of $V$ and $V'$ given in \eqref{2.07}, we see that we may restrict the sum over $l$ to be $N(l) \leq X^{1/2+\varepsilon}$ for any $\varepsilon>0$ in the error term above.  Note moreover from Lemma \ref{lemma Msize} and \eqref{Ikdef} that we have for any $m \in \mathcal M$,
\begin{align*}	
    \log N(m) \ll e^{\log_2 M}\log M\log_2 M(\log M)^{1+o(1)} \ll M^{\varepsilon}.
\end{align*}    
    It follows that
\begin{align}	
\label{Rest}
    R \ll X^{1/2+\varepsilon}\log( N(lmn)+2) \ll X^{1/2+\varepsilon}M^{\varepsilon}.
\end{align}        
    
    We further set $M= X^{\frac{1}{4}-5\varepsilon}$ for some $0<\varepsilon<1/20$ and observe that $|\mathcal{M}|\leq M$ from Lemma \ref{lemma Msize}. We then deduce that
\begin{align*}	
    \sum_{m, n\in \mathcal{M}} \sum_{l\odd } \frac{1}{\sqrt{N(l)}}R \ll X^{1/2+\varepsilon}M^2\sum_{\substack{l\odd \\ N(l) \leq X^{1/2+\varepsilon}} } \frac{1}{\sqrt{N(l)}} \ll X^{3/4+\varepsilon}M^2. 
\end{align*}
   We deduce from the above that
\begin{align}	
\label{S1exp}
\begin{split}
    \mathcal{S}_{1}=& {\widehat \Phi}(1)X\sum_{m, n \in \mathcal{M}}\sum_{\substack{ l \odd \\ lmn =\square}} \frac{1}{\sqrt{N(l)}}\int_{1}^{2}V\left(\frac{N(l)}{\sqrt{Xt}} \right) dt +O\Big(X^{3/4+\varepsilon}M^2\Big).  
\end{split}
\end{align}
  
  In view of \eqref{gcdprod} and keeping in mind that $V(x) \geq 0$, we may in the main term of \eqref{S1exp} keep only those satisfying
\[
l = \frac{[m,n]_{K}}{(m,n)_{K}}.
\]
  Using the relation given in \eqref{gcdprod} again, we see that 
\begin{align*}
%%\label{Dlower1}
\begin{split}
    \mathcal{S}_{1} \geq &  {\widehat \Phi}(1) X \sum_{m,n\in \mathcal M}\frac{N((m,n)_{K}^2)}{\sqrt{N(m)N(n)}} \int_1^2V\Big(\frac{N(m)N(n)}{N((m,n)_{K}^2)\sqrt{Xt}} \Big)dt  +O\Big(X^{3/4+\varepsilon}M^2\Big) \\
    \geq &  {\widehat \Phi}(1) X \sum_{\substack{m,n\in \mathcal M \\ N(m)N(n) \leq N((m,n)_{K}^2)X^{\varepsilon}}}\frac{N((m,n)_{K}^2)}{\sqrt{N(m)N(n)}} \int_1^2V\Big(\frac{N(m)N(n)}{N((m,n)_K^2)\sqrt{Xt}} \Big)dt  +O\big(X^{3/4+\varepsilon}M^2\big). 
\end{split}
\end{align*}
   
 Note that when $N(m)N(n) \leq N((m,n)_K^2)X^{\varepsilon}$, we have
$0< \frac{N(m)N(n)}{N((m,n)_K^2)\sqrt{Xt}}<1$ for $1 \leq t \leq 2$. We now apply the estimation $V(x)=1 + O\left(x^{\frac{1}{2}-\varepsilon}\right)$ given in \eqref{2.07} to deduce from the above that
\begin{align}
\label{S1lower}
\begin{split}
\mathcal{S}_{1} 
	 \ge & {\widehat \Phi}(1) X \sum_{\substack{m,n\in \mathcal M \\ N(m)N(n) \leq N((m,n)_K^2)X^{\varepsilon}}}\frac{N((m,n)_K^2)}{\sqrt{N(m)N(n)}}\big(1+O\big((X^{-1/2+\varepsilon})^{1/2-\varepsilon}\big)\big)+O\big(X^{3/4+\varepsilon}M^2\big).
\end{split}
\end{align}

  Analogue to the arguments in \cite[p. 25]{BT19} , we have, for each fixed \(m \in \mathcal M\),
\[
\sum_{n \in \mathcal M}\Big( \frac{N(m)N(n)}{N((m,n)_K^2)} \Big)^{1/3} \le \prod_{\substack{ \varpi \odd \\ N(\varpi) \le y_{\mathcal M}}}\Big(1+\frac{2}{N(\varpi)^{1/3}-1}\Big) \ll \exp \big(y_{\mathcal M}^{2/3}\big).
\]
  We now apply Rankin's trick to see that
\begin{align*}
   & \sum_{\substack{m,n\in \mathcal M \\ N(m)N(n) \leq N((m,n)_{K}^2)X^{\varepsilon}}}\frac{N((m,n)_{K}^2)}{\sqrt{N(m)N(n)}} \\
  &= \Big(\sum_{m,n\in\mathcal M } - \sum_{\substack{m,n\in\mathcal M \\ N(m)N(n) > N((m,n)_{K}^2)X^{\varepsilon}}}\Big)\frac{N((m,n)_{K}^2)}{\sqrt{N(m)N(n)}} 
     \\
    & \ge \sum_{m,n\in\mathcal M}  \frac{N((m,n)_{K}^2)}{\sqrt{N(m)N(n)}}  -X^{-\varepsilon/6}\sum_{m,n\in\mathcal M} \Big( \frac{N(m)N(n)}{N((m,n)_K^2)} \Big)^{1/3} \\
    & \gg S(\mathcal M) -X^{-\varepsilon/6}|\mathcal M| \exp \big(y_{\mathcal M}^{2/3}\big).
\end{align*}
   Note that by Lemma \ref{lemma Msize} we have $|\mathcal M| \leq M, \ y_{\mathcal M} \le  (\log M)^{1+o(1)}$. As $M= X^{\frac{1}{4}-5\varepsilon}$, we then apply Lemma \ref{GCD} to see that
\begin{align*}
%%\label{gcdlower}
\begin{split}
   & \sum_{\substack{m,n\in \mathcal M \\ N(m)N(n) \leq N((m,n)_{K}^2)X^{\varepsilon}}}\frac{N((m,n)_{K}^2)}{\sqrt{N(m)N(n)}} \ge M  \exp\bigg(\big(2\sqrt{2}+o(1)\big)\sqrt{\frac{\log M\log_3M}{\log_2M}}\bigg).
\end{split}
\end{align*}

  Using $M= X^{\frac{1}{4}-5\varepsilon}$ again, we deduce from \eqref{S1lower} and the above that 
\begin{align}
\label{S1lower2} 
\begin{split}
\mathcal{S}_{1} \ge & \left(1+o(1)\right){\widehat \Phi}(1) X M  \exp\bigg(\big(2\sqrt{2}+o(1)\big)\sqrt{\frac{\log M\log_3M}{\log_2M}}\bigg).
\end{split}
\end{align}   

  Next, we proceed to establish an upper bound of $\mathcal{S}_2$. Again, applying Lemma \ref{lemma logd} and arguing similar to those that lead to \eqref{Rest} imply that
\begin{align*}
	\mathcal{S}_{2}=& \sum_{q \odd }(\log N(q))R_q^2\Phi (\frac {N(q)}{X})=\sum_{m, n \in \mathcal{M}}  \sum_{q \odd }(\log N(q))\chi_{(1+i)^5q}(m n)\Phi (\frac {N(q)}{X})\\
	            =& {\widehat \Phi}(1)X \sum_{\substack{m, n \in \mathcal{M} \\ m n =\square}}1
	             +O\Bigg(X^{\frac 1 2 + \varepsilon} \sum_{\substack{m,n \in \mathcal{M}}}\log
  (N(mn)+2)\Bigg)\\
	            =& {\widehat \Phi}(1)X	\sum_{\substack{m, n \in \mathcal{M} \\ m n =\square}}1 + O\left(X^{1/2 +\varepsilon} M^{2+\varepsilon} \right) \\
  =& {\widehat \Phi}(1)X	\sum_{\substack{m, n \in \mathcal{M} \\ m n =\square}}1 + O\left(X^{1 -\varepsilon}  \right), 
\end{align*}
 where the last equality above follows from the observation that $M= X^{\frac{1}{4}-5\varepsilon}$. 

  We now apply Lemma \ref{lemma Msquaresize} to deduce from the above that 
\begin{align*}            
  \mathcal{S}_{2} \le \left(1+o(1)\right){\widehat \Phi}(1)XM.
\end{align*} 
   	
  We conclude from \eqref{maxlower}, \eqref{S1lower2} and the above that for sufficiently large $X$ and arbitrary small $\varepsilon>0$, 
\begin{align*}
\max_{\substack{q \odd \text{ \rm prime }\\X< N(q) \le 2X}}\Big|L(\frac 12, \chi_{(1+i)^5q})\Big | &\geq \frac{\mathcal{S}_1}{\mathcal{S}_2}\geq \exp\left(\left(2\sqrt{2} \cdot \sqrt{\frac 14-5\varepsilon}+o(1)\right)\sqrt{\frac{\log X\log_{3} X}{\log_{2} X}}\right).
\end{align*}
 This implies the estimation given in \eqref{Llowerbound} and hence completes the proof of Theorem \ref{main theorem 1}.
   
\hspace{0.1in}

\noindent{\bf Acknowledgments.} P. G. is supported in part by NSFC grant 12471003.

\bibliography{biblio}
\bibliographystyle{amsxport}

\vspace*{.5cm}

\end{document}